\documentclass{cimart}

\usepackage[utf8]{inputenc}

\usepackage{amsmath}
\usepackage{amsthm}
\usepackage{amssymb, mathtools, hyperref}
\usepackage{cleveref}
\usepackage{cite}

\def\Q{\mathbb Q}
\def\N{\mathbb N}

\def\R{\mathbb R}

\makeatletter
\newcommand*\bigcdot{\mathpalette\bigcdot@{.5}}
\newcommand*\bigcdot@[2]{\mathbin{\vcenter{\hbox{\scalebox{#2}{$\m@th#1\bullet$}}}}}
\newcommand{\dotp}{\boldsymbol{.}}
\makeatother

\title{Midy's \! Theorem \! in \! non-integer \! bases \!
and \! divisibility \! of \! Fibonacci \! numbers}

\author{Zuzana Mas\'akov\'a and Edita Pelantov\'a}

\authorinfo[
    Z. Mas\'akov\'a]{FNSPE, Czech Technical University in Prague, Czech Republic}{%
    zuzana.masakova@fjfi.cvut.cz
    }

\authorinfo[
    E. Pelantov\'a]{FNSPE, Czech Technical University in Prague, Czech Republic}{%
    edita.pelantova@fjfi.cvut.cz
    }    

\abstract{Fractions $\frac{p}{q} \in [0,1)$ with prime denominator $q$ written in decimal have a curious property described by Midy's Theorem, namely that two halves of their period (if it is of even length $2n$) sum up to $10^n-1$.  A number of results generalise Midy's theorem to expansions of  $\frac{p}{q}$ in different integer bases, considering non-prime denominators, or dividing the period into more than two parts. We show that a similar phenomena can be studied even in the context of numeration systems with non-integer bases, as introduced by Rényi. First we define the Midy property for a general real base $\beta >1$ and derive a necessary condition for validity of the Midy property. For $\beta =\frac12(1+\sqrt5)$ we characterize prime denominators $q$, which satisfy the property.}



\VOLUME{33}
\YEAR{2025}
\ISSUE{2}
\NUMBER{1}
\DOI{https://doi.org/10.46298/cm.12840}

\begin{document}



\section{Introduction}

The number $\frac{3}{7}$ in decimal system has a purely periodic expansion, namely $0\dotp(428571)^\omega$. 
Note that the first half of the period $428$ and the second half $571$ sum up to $999$. Similar behaviour appears with the fraction $\frac{18}{19}=0\dotp(947368421052631578)^\omega$. 
The sum of the two halves of the period is  $947368421+ 
052631578 = 999999999$. Here we always consider the period of minimal length.

According to Dickson~\cite{Dickson66}, this phenomenon for fractions of the form $\frac{1}{q}$ with prime denominator was observed experimentally already by Goodwyn in 1802~\cite{Goodwyn1802}. 

The first proof of this fact was probably given 
 in 1836 by a French college mathematics professor {\'E}tienne Midy in his privately published treatise on the properties of numbers and periodic decimal fractions~\cite{Midy1836}. Nowadays, under the name Midy's theorem, one usually finds the following result, although in Midy's text, one can actually find methods to show stronger properties of decimal fractions. 

\begin{theorem} Let  $q>5$ be a prime number. If a rational number  $\frac{p}{q} \in (0,1) $ has the minimal period of even length then the sum of the first and the second half of the period is a number whose expansion in the decimal system uses only the digit $9$. 
\end{theorem}

One can hardly cherish expectations of some theoretical consequences of this theorem, yet alone a down to earth application. Nevertheless, this result can serve for the general public as an illustration that mathematics can simply be fun. Proofs and generalisations of this theorem have been for decades a source of amusement of many, both mathematical amateurs and professionals. A historical survey can be found in~\cite{Shrader-Frechette78} and later~\cite{Ross2010}.

A nice presentation of Midy's theorem using group-theoretical proofs is given by Leavitt~\cite{Leavitt67}, who calls the phenomenon the nines-property and gives a criterion to decide about the parity of the period-length of the fraction $\frac{p}{q}$ in the decimal system in terms of quadratic residues. 
Leavitt also shows a sufficient condition for a fraction with non-prime denominator to have the nines-property.



A number of authors have focused on generalisations of Midy's theorem to considering fractions with non-prime denominators,
cutting the period of the fraction into more than two blocks of equal length, and translating the problem into non-decimal number systems with integer base $b \in \N$, see e.g.~\cite{Ginsburg2004,Gupta2005,Lewittes2007,Martin2007}. 

The aim of this contribution is to give a first glimpse to similar phenomena that appear when looking at fractions in numeration systems with non-integer base.
In 1957 A.~R\'enyi~\cite{Renyi57} introduced positional systems where the role of base is played by any real $\beta >1$. A representation of a given positive number $x$ in the form $x =\sum_{k= -\infty}^N x_k\beta^k$, with $x_k \in \N$, is found by the greedy algorithm and is called the $\beta$-expansion of $x$. 
The greedy algorithm produces digits in the set
 $\mathcal{D}=\{0,1, \ldots, \lceil{\beta\rceil}-1\}$. In case that  the base $\beta$ is not an integer, some combinations of digits do not appear in the $\beta$-expansion of any positive number $x$. One can characterize the strings of digits admissible for $\beta$-expansions using lexicographic comparison with the so-called quasigreedy $\beta$-expansion of the number $1$, usually denoted by $d_\beta^*(1) = t^*_1t^*_2 t^*_3\cdots$, see~\cite{Parry60}. The  string $d_\beta^*(1)$ is composed of digits over the set  $\mathcal{D}$, and it is the lexicographically largest string with 
 infinitely many non-zero digits such that  $1 = \sum_{k=1}^{+\infty}t_k^* \beta^{-k}$. For example, if $\beta=10$, then $d_\beta^*(1) = 9^\omega$.

Our special attention is given to the  numeration system with the golden ratio base $\tau = \frac{1+\sqrt{5}}{2}\approx 1\dotp 618$.  The digits in this system are only  $0$ and $1$ and the base satisfies $\tau^2 = \tau +1$.   Consequently, the quasigreedy $\tau$-expansion of $1$ is $d_\tau^*(1) = (10)^\omega$.  

It is known~\cite{Schmidt80} that every rational number  $\frac{p}{q} \in (0,1)$ has a purely periodic $\tau$-expansion. For example, $\frac{3}{7}$ has the $\tau$-expansion  $0\dotp(0100001001010010)^\omega$, whose period-length is equal to 16. The first half of the period $01000010$ represents the number $x_1 = \tau^6 + \tau$, whereas the second half $01010010$  gives $x_2=\tau^6 + \tau^4 +\tau$. Using $\tau^{k+2}= \tau^{k+1}+\tau^k$, we easily derive that $x_1+x_2 =  \tau^7 + \tau^5 +\tau^3+\tau = \tau^8-1$ and hence, the $\tau$-expansion of the sum is equal to $10101010$. 
Note that this string is a prefix of the quasigreedy expansion $d_\tau^*(1)$. 
In case of the classical decimal system, the sum of the two halves of the period is a prefix of the string $9^\omega$, which is the quasigreedy expansion of unity for the decimal base. 

In both the presented examples for the bases $\beta = 10$ and $\beta = \tau$ it holds  that if a period of a fraction is of even length, say $2n$, then the sum of the two halves has the value $\beta^n-1$. 

This observation suggests how the `nines-property' given for the decimal system can be extended to systems with arbitrary real base $\beta >1$, see Definition~\ref{d:MidyProp}. In Section~\ref{sec:nutna} we show necessary condition so that a fraction with denominator $q$ has the Midy property in a base $\beta >1$. In Section~\ref{sec:postacujici} we study sufficient conditions for the golden ratio base $\tau$. With the use of divisibility properties of Fibonacci numbers, we characterize the prime denominators $q\in \mathbb{N}$, for which an analogy of  Midy's theorem holds in base $\tau$, see Section~\ref{sec:primes}.

\section{Preliminaries}\label{sec:preliminaries}

Given a real number $\beta>1$, one can obtain the $\beta$-expansion of a positive real number $x$ by the greedy algorithm: Find $k$ such that 
$\beta^k\leq x < \beta^{k+1}$, $r_k=x$, and for $i\leq k$ repeat: $x_i:=\lfloor r_i/{\beta^i}\rfloor$, $r_{i-1}:=x-x_i\beta^i$. Then 
$$
x=\sum_{i\leq k} x_i\beta^i, \quad x_j\in
{\mathcal D}:=\{k\in \N: k< \beta\},
$$
and for every $j\leq k$, 
we have $\sum_{i\leq j}x_i\beta^i < \beta^{j+1}$. 
For the $\beta$-expansion of $x$, we write
$$
(x)_{\beta} = \begin{cases}
x_kx_{k-1}\cdots x_0\dotp x_{-1}x_{-2}\cdots &\text{ if }k\geq 0,\\
0\dotp 0^{-k-1}x_{k}x_{k-1}\cdots & \text{ if }k<0.
\end{cases}
$$
In case $x\in(0,1)$, the $\beta$-expansion of $x$ can be defined by the  $\beta$-transformation given by $T_\beta:[0,1]\to[0,1)$, $T_\beta(x)=\beta x - \lfloor\beta x\rfloor$, setting
$(x)_\beta = 0\dotp x_1x_2x_3\cdots$ where $x_i=\lfloor\beta T^{i-1}(x)\rfloor$. Note that for every $n\in\N$, we have
$$
T^n(x)=0\dotp x_{n+1}x_{n+2}x_{n+3}\cdots = \big(x-\sum_{k=1}^nx_k\beta^{-k}\big)\beta^n.
$$
In general, not all combinations of digits in ${\mathcal D}$ appear in a $\beta$-expansion. The sequences of digits which are admissible as $\beta$-expansions are described by the lexicographic condition, using the so-called quasigreedy expansion of 1, denoted $d_\beta^*(1)=t_1^*t_2^*t_3^*\cdots$, defined by
$\lim_{x\to1-} (x)_\beta= 0\dotp t_1^*t_2^*t_3^*\cdots$, where the limit is considered in the product topology. The theorem by Parry~\cite{Parry60} then says that $0\dotp x_1x_2x_3\cdots$ with $x_i\in\N$, is a $\beta$-expansion of some $x\in(0,1)$ if and only if for every $i\geq 1$, we have $x_ix_{i+1}x_{i+2}\cdots\prec d_\beta^*(1)$, where $\preceq$ stands for standard lexicographic order on strings.

The so-called $\beta$-integers are real numbers whose $\beta$-expansion has no non-zero digits to the right of the fractional point, 
$$
{\mathbb Z}_\beta=\{x\in \R : \big(|x|\big)_\beta= x_nx_{n-1}\cdots x_1x_0\dotp 0^\omega\}.
$$

\begin{example}
Let $\tau=\frac12(1+\sqrt5)\approx 1\dotp 618$ be the golden ratio.   
By the greedy algorithm, we can calculate
the $\tau$-expansion of 2. We have $\tau^1\leq 2<\tau^2$, thus $k=1$,
$$
\begin{array}{ll}
r_1=2, & x_1=\lfloor 2/\tau^1 \rfloor = 1,\\[1mm]
r_0=2-\tau,  & x_0=\lfloor (2-\tau)/\tau^0 \rfloor = 0,\\[1mm]
r_{-1}=2-\tau, & x_{-1} = \lfloor (2-\tau)/\tau^{-1}\rfloor = 0,\\[1mm]
r_{-2}=2-\tau, & x_{-2} = \lfloor (2-\tau)/\tau^{-2}\rfloor = 1.
\end{array}
$$
Since $r_{-3}=2-\tau-\tau^{-2}=0$, we have $x_i=0$ for every $i\leq -3$ and $(2)_\tau=10\dotp 010^\omega$, where by $0^\omega$ we mean infinite repetition of the digit 0. As usual in the decimal system, we can omit the suffix 
$0^\omega$.

Let us now compute the $\tau$-expansion of $1/2$ using the $\tau$-transformation. We have
$$
\begin{array}{ll}
x_1=\lfloor \frac\tau2 \rfloor=0, &
T_\tau(1/2) = \frac\tau2 - \lfloor \frac\tau2 \rfloor = \frac\tau2,\\[2mm]
x_1=\lfloor \frac{\tau^2}2 \rfloor=1, &
T^2_\tau(1/2) = \frac{\tau^2}2 - \lfloor \frac{\tau^2}2 \rfloor =\frac1{2\tau},\\[2mm]
x_2=\lfloor \frac12 \rfloor = 0, & T^3_\tau(1/2) = \frac12. 
\end{array}
$$
Since $T^3_\tau(1/2)=T^0_\tau(1/2)$, we have $T^{n+3}_\tau(1/2)=T^n_\tau(1/2)$, and thus $(1/2)_\tau=0\dotp (010)^\omega$.
Similarly, one can obtain the purely periodic $\tau$-expansion of $\frac37$ as it was mentioned in the introduction,
$(3/7)_\tau=0\dotp (0100001001010010)^\omega$.

Note that for all $(2)_\tau$, $(1/2)_\tau$, and $(3/7)_\tau$, the string of digits does not contain two consecutive digits equal to 1. This is not a coincidence. For, the quasigreedy expansion of 1 satisfies $d_\tau^*(1)=(10)^\omega$. The Parry lexicographic condition says that  a $\beta$-expansion has only digits in the set $\{0,1\}$, does not contain the string $11$ and does not end with the tail $(01)^\omega$. 

With this in hand, we can find the first few non-negative
$\tau$-integers. Their $\tau$-expansions are 
$$
0,\ 1,\ 10,\ 100,\ 101,\ 1000,\ 1001,\ 1010,\ 10000,\ \dots
$$
They have values
$$
0,\ 1,\ \tau,\ \tau^2,\ \tau^2+1,\ \tau^3,\ \tau^3+1,\ \tau^3+\tau,\ \tau^4,\ \dots.
$$
\end{example}


\begin{definition} \label{d:MidyProp}
Let $\beta >1$.  We say that $q \in \mathbb{N}$ {\it  has the Midy property in base $\beta$}, if there exists a positive integer $p <q$ coprime with $q$ such that 

\noindent
$\bullet$
the $\beta$-expansion of $\frac{p}{q}$ is purely periodic 
$\bigl(\frac{p}{q}\bigr)_\beta =0\dotp (c_1c_2\cdots c_{2n})^\omega$ where $2n$ is the length of the shortest period; and

\noindent
$\bullet$  
$x+y = \beta^n-1$, where $x,y$ are $\beta$-integers with 
$\beta$-expansions $(x)_\beta=c_1c_2\cdots c_{n}$ and $(y)_\beta = c_{n+1}c_{n+2}\cdots c_{2n}$,  respectively. 

\noindent
The number $p$ is then said to {\it testify}
to the Midy property of $q$ in base $\beta$.
\end{definition}

From the above given examples, we see that in base $\tau$ the number $q=7$ has the Midy property whereas the number  $q=2$ has not. For, the only fraction $\frac{p}{q}$ in the interval $(0,1)$ with denominator $q=2$ is $\frac12$ with the $\tau$-expansion $(\frac12)_\tau=0\dotp(010)^\omega$ of odd length.

\begin{remark}
Note that directly from the definition it follows that if an integer $q$ has the Midy property in base $\beta$, then the base is an algebraic integer. Indeed, we have
$$
\beta^n-1 = x+y \quad \text{ where } x=\sum_{i=1}^nc_i\beta^{n-i} \text{ and } y=\sum_{i=1}^nc_{n+i}\beta^{n-i}\,,
$$
which shows that $\beta$ is a root of a monic polynomial with integer coefficients.
\end{remark}

Our aim is to search for bases in which infinitely many fractions $\frac{p}{q}$ satisfy the Midy property. The crucial point is to have infinitely many positive fractions $\tfrac{p}{q}<1$, with purely periodic expansion. Pure periodicity in non-integer bases was studied already by Schmidt~\cite{Schmidt80}, later by Hama and Imahashi~\cite{HI97}, Akiyama~\cite{A98}, Adamczewski et al.~\cite{AFSS10} and others.
Let us summarize the results. 

Denote $\gamma(\beta)$ supremum of real numbers $\gamma$ such that every $x\in[0,\gamma)\cap\Q$ has a purely periodic $\beta$-expansion.
Based on the results of Schmidt~\cite{Schmidt80}, Akiyama~\cite{A98} has shown that if $\gamma(\beta)>0$, then $\beta$ is a Pisot unit. 
Recall that an algebraic integer $\beta=\beta^{(1)}>1$ is a Pisot number of degree $d$, if its minimal polynomial $f\in{\mathbb Z}[X]$ is of degree $d$, and the other roots $\beta^{(i)}$, $i=2,\dots,d$ of $f$, called the algebraic conjugates of $\beta$, are in modulus smaller than $1$.  
The number $\beta$ is an algebraic unit, if its norm, $N(\beta)=\prod_{i=1}^d\beta^{(i)}$ is equal to $\pm1$.
There are two classes of quadratic Pisot units, namely the roots $\beta>1$ of the polynomials
\begin{eqnarray}
X^2-mX-1,&\quad m\geq 1,\label{eq:quadrPisot-1}\\
X^2-mX+1,&\quad m\geq 3.\label{eq:quadrPisot+1}
\end{eqnarray}

From Schmidt~\cite{Schmidt80}, it follows that if $\beta$ is a root of~\eqref{eq:quadrPisot-1}, then $\gamma(\beta)=1$, i.e.\ every fraction in the interval $(0,1)$ has a purely periodic $\beta$-expansion. On the other hand, for roots of~\eqref{eq:quadrPisot+1}, it is shown in~\cite{HI97} that no fraction has a purely periodic $\beta$-expansion, and hence $\gamma(\beta) = 0$.

Let
$f(X) = X^d - c_{d-1}X^{d-1} - c_{d-2}X^{d-2} - \cdots - c_1X-c_0\in \mathbb{Z}[X]$ be the minimal polynomial of $\beta$, the companion matrix of $\beta$ is defined as
$$
C =\left(\begin{array}{ccccc}
     0&0& \cdots &0 & c_0\\
   1&0&\cdots &0&c_1\\
   \vdots &&\ddots&& \vdots\\
   0&0& \cdots & 0 &c_{d-2}\\
   0&0& \cdots &1&c_{d-1}
\end{array}\right)
$$
The spectrum of $C$ is formed by the roots of $f$, in particular, the determinant of $C$ is equal to the norm of $\beta$, $\det C= N(\beta)$. Denote $\mathbf{v} = (1, \beta, \beta^2, \ldots, \beta^{d-1})^T$. Then $\mathbf{v}$ is a left eigenvector of $C$ corresponding to the eigenvalue $\beta$,
$$
\mathbf{v}^TC = \beta\mathbf{v}^T.
$$

\begin{example}
The minimal polynomial of the golden ratio $\tau = \frac12(1+\sqrt{5})$ is given by $f(X)=X^2-X-1$, the companion matrix of $\tau$ is $C = \left(\begin{smallmatrix}
    0&1\\
    1&1
\end{smallmatrix}\right)$. The algebraic conjugate of $\tau$ is $\tau'=\frac12(1-\sqrt5)=-\frac1\tau\sim-0.618$, and thus $\tau$ is a Pisot number. 
\end{example}

\section{Necessary  condition}\label{sec:nutna}

\begin{lemma}\label{lem:IFF} 
Let $q \in \N, q>2$ and $\beta >1$. Then $q$ satisfies the  Midy property for $\beta$ if and only if there exists $p\in \N$, $0<p<q$, $p$ coprime with $q$ and  $N \in \N$ such that 
\begin{equation}\label{eq:doplnek}
T^N(\tfrac{p}{q}) = \tfrac{q-p}{q} \quad \text{and} \quad T^N(\tfrac{q-p}{q}) = \tfrac{p}{q}\,.
\end{equation}
\end{lemma}

\begin{proof} 
Suppose that a number $z\in(0,1)$ has purely periodic expansion with period of even length $2n$, say
$\bigl(z\bigr)_\beta =0\dotp(c_1c_2\cdots c_{2n})^\omega$, i.e. 
$$
\bigl(T^n ( z )\bigr)_\beta = 0\dotp (c_{n+1}c_{n+2} \cdots c_{2n}c_1c_2\cdots c_{n})^\omega\qquad \text{and} \qquad T^{2n} ( z ) =z,
$$
This can be written using the $\beta$-integers
$x :=c_1\beta^{n-1} + c_2\beta^{n-2}+\cdots + c_{n-1}\beta + c_n$  and   $y :=c_{n+1}\beta^{n-1} + c_{n+2}\beta^{n-2}+\cdots + c_{2n-1}\beta + c_{2n}$, as
$$
z = \frac{x\beta^n + y}{\beta^{2n}-1}  \quad \text{and}  \quad T^n ( z ) =  \frac{y\beta^n + x}{\beta^{2n}-1}.
$$
We derive
\begin{equation}\label{eq:x+y}
z+T^n(z) = \frac{(x+y)(\beta^n+1)}{\beta^{2n}-1} = 
\frac{x+y}{\beta^{n}-1}.    
\end{equation}

Let $q$ satisfy the Midy property in base $\beta$. This means that we have~\eqref{eq:x+y} for some $z=\frac{p}{q}$ where $p$ is coprime with $q$, and, moreover, $x+y=\beta^n-1$. Equation~\eqref{eq:x+y} gives $\tfrac{p}{q} +  T^n ( \tfrac{p}{q} ) = 1$. Hence, $T^n ( \tfrac{p}{q} ) =  \tfrac{q-p}{q}$ and   $T^n ( \tfrac{q-p}{q} )=T^{2n} ( \tfrac{p}{q} ) = \tfrac{p}{q}$. It suffices to set $N=n$. 

For the opposite implication, assume that Equation~\eqref{eq:doplnek} is satisfied for some $N\in\N$ and 
$p\in\{1,2,\dots,q-1\}$ coprime with $q$. 
Denote $d$ the shortest period of the $\beta$-expansion of $\frac{p}{q}$. Obviously, $d$ divides $2N$. If $d$ divides $N$, then $\frac{p}{q}=T^N(\frac{p}{q})=\frac{q-p}{q}$, which happens only for $\frac{p}{q}=\frac12$. Since $q>2$ and $p$ is coprime with $q$, this is not possible.

Thus $2N=kd$ for some odd $k$, i.e.\ the minimal period-length $d$ is even. 
We have $N=kd/2 = (k-1)d + d/2$.
Substituting into~\eqref{eq:doplnek}, we obtain
$$
\tfrac{q-p}{q}=T^N\big(\tfrac{p}{q}\big) = T^{d/2}\Big(T^{(k-1)d}\big(\tfrac{p}{q}\big)\Big) = T^{d/2}\big(\tfrac{p}{q}\big),
$$
and consequently also
$$
T^{d/2}\big(\tfrac{q-p}{q}\big)= 
T^{d/2}\Big(T^{d/2}\big(\tfrac{p}{q}\big)\Big) = T^{d}\big(\tfrac{p}{q}\big) = \tfrac{p}{q}.
$$
Therefore~\eqref{eq:doplnek} is satisfied also for $N=d/2$.

Therefore the $\beta$-expansion of $z=\frac{p}{q}$ is purely periodic of length $d$. Thus~\eqref{eq:x+y} is satisfied with $n=d/2$. Combining with~\eqref{eq:x+y}, we derive 
$z+T^{n}(z)=1 = \frac{x+y}{\beta^{n}-1}$, whence $x+y=\beta^n-1$. 
This concludes the proof.
\end{proof}

\begin{theorem}\label{t:Nutne} Let $C\in \mathbb{Z}^{d\times d}$  be the companion matrix of an algebraic integer $\beta >1$ of degree $d$.  
If $q\in \N$, $q > 2$, has the Midy property in base $\beta$, then there exists a positive integer  $N$ such that $C^N \equiv -I \mod q$.
\end{theorem}

\begin{proof}  
Denote $\mathbf{v} = (1, \beta, \beta^2, \ldots, \beta^{d-1})^T$.  
If $x \in \frac{1}{q}\mathbb{Z}[\beta]$, then there exists a unique integer vector $ \mathbf{a}(x)= (a_0, a_1, \ldots, a_{d-1})^T \in \mathbb{Z}^d$ such that
$$
x = \tfrac{1}{q}  \ \mathbf{v}^T \mathbf{a}(x). 
$$
Let  $d=\lfloor \beta x\rfloor$ be the first digit in the $\beta$-expansion of $x$. Then 
\begin{equation}\label{eq:vectorModulo}
T(x) = \beta x - d = \tfrac{1}{q}  \bigl( \mathbf{v}^T C\mathbf{a}(x) - qd  \bigr)  = \tfrac{1}{q}   \mathbf{v}^T \bigl(C\mathbf{a}(x) - qd \mathbf{e}_1 \bigr) \in \tfrac{1}{q}\mathbb{Z}[\beta], 
\end{equation}
where by $\mathbf{e}_i \in \mathbb{R^d}$ we mean the $i^{th}$ column of the identity matrix $I \in \mathbb{R}^{d\times d}$, in other words,  $\mathbf{e}_i$ denotes the $i^{th}$ vector of the canonical base of the vector space $\mathbb{R}^d$. 

From~\eqref{eq:vectorModulo}, we see that the set  $\frac{1}{q}\mathbb{Z}[\beta] \cap [0,1)$  is closed under the transformation $T$.  
Equation \eqref{eq:vectorModulo} implies that
\begin{equation}\label{eq:vectorProT}
\mathbf{a}(T(x)) = C\mathbf{a}(x) - qd \mathbf{e}_1  \equiv C\mathbf{a}(x)  \mod q.
\end{equation} 

 Clearly, $\mathbf{a}(\tfrac{p}{q}) = p\mathbf{e}_1$ and $\mathbf{a}(\tfrac{q-p}{q}) =(q- p)\mathbf{e}_1$.  
Applying~\eqref{eq:vectorProT} to
Equation~\eqref{eq:doplnek} we obtain 
$$
\mathbf{a}\Bigl(T^N\bigl(\tfrac{p}{q}\bigr)\Bigr) \equiv C^N\mathbf{a}\bigl(\tfrac{p}{q}\bigr) \equiv C^N p\mathbf{e}_1 \equiv - p\mathbf{e}_1 \mod q. 
$$
Since $p$ and $q$ are coprime, we have derived 
$C^N \mathbf{e}_1 \equiv - \mathbf{e}_1 \mod q.$ 
In order to finish the proof, we need to verify that
$C^N \mathbf{e}_i \equiv - \mathbf{e}_i \mod q$  
holds for all vectors $\mathbf{e}_i$, $i =2,3, \ldots, d$.   For this purpose, we show by induction the following claim:

\medskip
\centerline{\emph{For any $n \in \N$ and any $i = 2,\ldots, d$ one has $C^n \mathbf{e}_{i} = C^{n+1}\mathbf{e}_{i-1}$.}}
\medskip

Indeed, if $n=0$, then the claim can be checked directly  from the definition of the companion matrix $C$.  Assume that the statement is valid for $n \in \N$. Then with the use of the induction hypothesis, we have $C^{n+1}\mathbf{e}_{i} = C C^n\mathbf{e}_{i} = C C^{n+1} \mathbf{e}_{i-1} = C^{n+2}\mathbf{e}_{i-1}$, i.e.\ the statement is valid for $n+1$, as well.  

\medskip

Combining the claim with the relation  $C^N\mathbf{e}_1\equiv -\mathbf{e}_1 \!\!\mod q$  we obtain 
\[
C^N\mathbf{e}_2 \equiv C^{N+1}\mathbf{e}_1 \equiv C C^N \mathbf{e}_1 \equiv C(-\mathbf{e}_1) \equiv -\mathbf{e}_2 \!\!\mod q.
\]
We proceed analogously to show 
$C^N\mathbf{e}_i \equiv - \mathbf{e}_i \!\!\mod q$, for all $i$. This proves that $C^N\equiv -I\!\!\mod q$. 
\end{proof}

\begin{remark}\label{rem:Tribonacci} 
If $\beta >1$ is an algebraic  number of an odd degree $d$ and $\beta$ has norm $N(\beta)=1$, then no integer $q$ satisfies the Midy property in base $\beta$. For, $N(\beta)= 1 = \det C =  \det C^N $ and $\det (-I) = (-1)^d=-1$, the equality  $C^N \equiv -I \mod q$  cannot hold true.  In particular, no $q>2$ satisfies the Midy property in the Tribonacci base $\beta$ -  the positive root of the polynomial $X^3-X^2-X-1$. 
\end{remark}

\begin{remark} 
Let us mention that the necessary condition given in Theorem~\ref{t:Nutne} is not sufficient. As  counterexample consider $\beta>1$, the quadratic Pisot number with minimal polynomial $X^2-3X+1$. The companion matrix 
    $C=\left( \begin{smallmatrix}  0  &-1  \\
    1 & 3  \end{smallmatrix} \right)$ satisfies for $q=5$ that $C^5 \equiv - I \mod q$. 
    Nevertheless, it is known that no rational number in $(0,1)$ has a purely periodic $\beta$-expansion, and thus  $q=5$ does not satisfy the Midy property.
\end{remark}

\section{Sufficient condition for the base $\tau = \frac{1+\sqrt{5}}{2} $}
\label{sec:postacujici}

Our aim is to show that the necessary condition derived in Theorem \ref{t:Nutne} for an integer $q$ to satisfy the Midy property in base $\beta>1$ is also sufficient in case of $\beta$ being the golden ratio $\tau$.
Powers of the companion matrix $C$ of the golden ratio can be expressed using the well known Fibonacci sequence $(F_n)_{n\in\N}$, defined by 
$$
F_0=0, F_1=1, \text{ and } \ F_{n+2} = F_{n+1}+F_n \text{ for }n \in \N.
$$ 
It can be easily computed that for any positive exponent $N\in\N$, we have
$$
C^N = \left(\begin{smallmatrix}
    F_{N-1}& F_N\\
    F_{N}& F_{N+1}
\end{smallmatrix}\right).
$$

\begin{theorem}\label{t:staci} Let $C$ be the companion matrix of the golden ratio $\tau$, let $q, N \in \N$, $q>2$ and $N>1$. If $C^N \equiv -I \mod q$, then  
   $q$ has the Midy property in base $\tau$. Every $p\in \mathbb{N}$, $0<p<q$, testifies to the Midy property of $q$. 
\end{theorem}

\begin{proof} 
First we show that for any fraction  $x  \in \mathbb{Q} \cap (0,1)$ with denominator $q$ it holds that $T^N(x) = 1-x$. Then also $T^N(1-x) = 1-(1-x) = x$, hence $T^{2N}(x) = x$. By Lemma~\ref{lem:IFF}, this implies the statement.
 
Let $c_1, c_2, \ldots, c_N$ be the digits of the $\tau$-expansion of $x = \tfrac{p}{q}$ obtained by first $N$ iterations of the transformation $T$. Then 
\begin{equation}\label{eq:rozvojPQ}
(0,1)\ni T^N(\tfrac{p}{q}) =\left( \frac{p}{q} - \frac{c_1}{\tau} - \frac{c_2}{\tau^2} - \cdots - \frac{c_N}{\tau^N} \right)\tau^N = \frac{1}{q} \Bigl(p\tau^N - q\sum_{k=0}^{N-1}c_{N-k}\tau^k\Bigl).
\end{equation}
By induction, one can easily verify the following formula
connecting Fibonacci numbers and the golden mean,
\begin{equation}\label{eq:Binet}
\tau^k = \bigl(\tau+\tfrac{1}{\tau}\bigr)F_{k} +  \bigl(\tfrac{-1}{\tau}\bigr)^k, \text{ for }k\geq 0.    
\end{equation}
Substituting into~\eqref{eq:rozvojPQ}, one has
$$
T^N(\tfrac{p}{q}) = \tfrac{1}{q}\Bigr(\bigl(\tau+\tfrac{1}{\tau}\bigr) \bigl(\underbrace{ pF_N - q\sum_{k=0}^{N-1}c_{N-k}F_k}_{=:A}\bigr)  +  p\bigl(\tfrac{-1}{\tau}\bigr)^N - q  \underbrace{\sum_{k=0}^{N-1}c_{N-k}\bigl(\tfrac{-1}{\tau}\bigr)^k}_{=:B}\,\Bigr)
$$
The assumption $C^N \equiv -I \mod q$ implies that $F_N \equiv 0 \mod q$ and  $F_{N -1}\equiv -1 \mod q$.  Therefore $A \in \mathbb{Z}$ can be written in the form
$$
A=pF_N - q\sum_{k=0}^{N-1}c_{N-k}F_k =\ell q\quad\text{ for some }\ell\in  \mathbb{Z}.
$$ 
Consequently, we have the following estimate, 
\begin{equation}\label{eq:odhad}
T^N(\tfrac{p}{q}) = |T^N(\tfrac{p}{q})| > \bigl(\tau+\tfrac{1}{\tau}\bigr) |\ell| - \tfrac{1}{\tau^N}  -|B|.
\end{equation}

Recall that in the sequence of digits
$c_1, c_2, \ldots, c_N$, there are never two consecutive digits equal to 1. Therefore we have the estimate on $B=\sum_{k=0}^{N-1}c_{N-k}\bigl(\tfrac{-1}{\tau}\bigr)^k$,
$$
|B| \leq \left\{ \begin{array}{ll} \sum_{k=0}^\infty\tfrac{1}{\tau^{2k}} - \tfrac{1}{\tau^{N+1}}\sum_{k=0}^\infty\tfrac{1}{\tau^{2k}} =\tau - \frac{1}{\tau^{N}}, \quad & \text{if\ \ } c_N =1,\\ 
&\\
\sum_{k=0}^\infty\tfrac{1}{\tau^{2k+1}} = 1, \quad & \text{if\ \ } c_N =0.
\end{array}\right.
$$

We will now show that $A= \ell q = 0$. Assume that the opposite is true, i.e.\ $|\ell|\geq 1$. 

In case $c_N =0$, estimate~\eqref{eq:odhad} gives   $T^N(\tfrac{p}{q}) > \bigl(\tau+\tfrac{1}{\tau}\bigr) - \tfrac{1}{\tau^N} - 1 = \tfrac{2}{\tau} -\tfrac{1}{\tau^N} >1$, which is a contradiction.

If $c_N =1$, then the digit $c_{N+1}=0$, and thus $T^{N+1}(\tfrac{p}{q}) =\tau  T^N(\tfrac{p}{q}) <1$, i.e.\  $T^N(\tfrac{p}{q}) < \frac{1}{\tau}$.  If $\ell \neq 0$, we obtain  $T^N(\tfrac{p}{q}) > \bigl(\tau+\tfrac{1}{\tau}\bigr) - \tfrac{1}{\tau^N} - \tau + \tfrac{1}{\tau^N} =  \frac{1}{\tau}$, which is again a contradiction. 
Thus we have derived that $A =0$. 

Now we use another expression for the powers 
of the golden ratio using Fibonacci numbers, namely
\begin{equation}\label{eq:tau^k}
\tau^k=F_k\tau + F_{k-1},
\end{equation}
which holds for $k\geq 0$  defining $F_{-1}=1$.
We substitute for $\tau^k$ into~\eqref{eq:rozvojPQ}, to obtain
$$
T^N(\tfrac{p}{q}) = \tfrac{1}{q}\Bigr(\tau\bigl(\underbrace{ pF_N - q\sum_{k=0}^{N-1}c_{N-k}F_k}_{=A}\bigr)  + pF_{N-1} - q\underbrace{\sum_{k=0}^{N-1}c_{N-k}F_{k-1}}_{\in \mathbb{Z}}\Bigr).
$$
Since $A = 0$ and $F_{N-1} = -1 \mod q$, there exists $n\in \mathbb{Z}$ such that  
$$
T^N(\tfrac{p}{q}) = \tfrac{1}{q} \bigl(-p + n q  \bigr) = n -\tfrac{p}{q}
$$
Since both $T^N(\tfrac{p}{q}) \in (0,1)$  and $\frac{p}{q} \in (0, 1) $,  necessarily $n=1$ and thus $T^N(\tfrac{p}{q})= \tfrac{q-p}{p}$ as we wanted to show.
\end{proof}

\begin{remark} Let  $p, q\in \N$,\ $1\leq p< q$, $p$ coprime with $q$.  If $p$ testifies to the Midy property of $q$ in base $\tau$, then the minimal period $d$ of the $\tau$-expansion of $\frac{p}{q}$ is even and the equation $C^N \equiv -I \mod q$,  is satisfied also for $N = \frac{d}{2}$, see the proof of Lemma~\ref{lem:IFF}. In particular, $\det C^{d/2} = (-1)^{d/2} \equiv  \det(-I) \equiv  1 \mod q$. Hence $\frac{d}{2}$ is even, i.e.\ the minimal period $d$ is divisible by $4$. The sum of the two halves of the $\tau$-expansion of $\frac{p}{q}$ is equal to $\tau^{d/2}-1 = \sum_{k=1}^{d/4}\tau^{2k-1}$ and thus  its $\tau$-expansion is ${(10)^{d/4}}$, which is a prefix of $d_\tau^*(1)=(10)^\omega$.   
    
\end{remark}

\begin{example}\label{ex:q=3} 
The number
$q =3$ satisfies the Midy property in base $\tau$, as 
$$
C^4 = \left(\begin{smallmatrix}
    F_{3}& F_4\\
    F_{4}& F_{5}
\end{smallmatrix}\right) = \left(\begin{smallmatrix}
    2& 3\\
    3& 5
\end{smallmatrix}\right)\equiv -I \mod 3.
$$ 
Indeed,  $\bigl(\tfrac{1}{3}\bigr)_\tau = 0\dotp (00101000)^\omega$. The first half of the period  $0010$  represents the number $\tau$, the second half   $1000$ represents the number $\tau^3$. Their sum is 
\[
\tau^3 +\tau = (\tau^3 + \tau + 1) - 1 = (\tau^3 +\tau^2) - 1 = \tau^4-1.
\]
The $\tau$-expansion of the sum is the string $1010$ and it is a prefix of  $d^*_\tau(1) = (10)^\omega$. 

Similarly, $q =5$ satisfies Midy property in base $\tau$, as 
$$
C^{10} = \left(\begin{smallmatrix}
    F_{9}& F_{10}\\
    F_{10}& F_{11}
\end{smallmatrix}\right) = \left(\begin{smallmatrix}
    34& 55\\
    55& 89
\end{smallmatrix}\right)\equiv -I \mod 5.
$$ 
Indeed,  $\bigl(\tfrac{1}{5}\bigr)_\tau = 0\dotp (00010010101001001000)^\omega$. The first half of the period  $0001001010$  represents the number $\tau^6+\tau^3+\tau$, the second half   $1001001000$ represents the number $\tau^9+\tau^6+\tau^3$. Their sum is $\tau^9 +2\tau^6+2\tau^3+\tau = \tau^{10}-1$. The $\tau$-expansion of the sum is the string $1010101010$ and it is a prefix of  $d^*_\tau(1) = (10)^\omega$. 
\end{example}

\begin{corollary}\label{co:divisors} Let $q$ satisfy the Midy property in base $\tau$. If $d\in \N, d>2$, is a divisor of $q$, then $d$ has the Midy property in $\tau$ as well.  
\end{corollary}

\begin{proof} 
By Theorem \ref{t:Nutne}, there exists a positive integer $N$ such that $C^N \equiv -I \mod q$. As $d$ is divisor of $q$, necessarily  $C^N \equiv -I \mod d$. By Theorem \ref{t:staci},  $d$ has the Midy property in base $\tau$. 
\end{proof}

\begin{corollary}\label{cor:FibonacciLicheSude} 
Let  $q\in \N, q>2$. \begin{enumerate}
    \item If $q$ is a divisor of  $F_{2n-1}$ for some $n \in \N, n\geq 3$, then $q$ has the Midy property in~base $\tau$. 
    \item If $q$ is a multiple  of  $F_{2n}$ for some $n \in \N, n\geq 3$, then $q$ does not have the Midy property in base $\tau$. 
\end{enumerate}
\end{corollary}

\begin{proof} In view of Corollary~\ref{co:divisors} it suffices to show that $q=F_n$, $n\geq 5$ satisfies the Midy property if and only if $n$ is odd. 

Note that  $\det C = -1$ and thus $\det C^k = F_{k-1} F_{k+1} - F_k^2 = (-1)^k$. This fact together with the recurrence for the Fibonacci sequence implies that
\begin{equation}\label{eq:determinantModulo}
F_{k+1}^2 \equiv (-1)^k \mod F_k \qquad \text{and}\qquad  
F_{k+1} \equiv F_{k-1} \mod F_{k}.
\end{equation} 
Thus
\begin{equation}\label{eq:MaticeModulo}
C^k = \left(\begin{smallmatrix}
    F_{k-1}& F_k\\
    F_{k}& F_{k+1}
\end{smallmatrix}\right)\equiv \left(\begin{smallmatrix}
    F_{k+1}& 0\\
    0& F_{k+1}
\end{smallmatrix}\right) \mod F_k.
\end{equation}

\begin{enumerate}
\item 
Let $q =F_{2n-1}$.  Put $N = 2(2n-1)$. Using \eqref{eq:determinantModulo}  and \eqref{eq:MaticeModulo} we deduce 
$$
C^{N} =C^{2n-1} C^{2n-1} \equiv \left(\begin{smallmatrix}
    F_{2n}^2& 0\\
    0& F_{2n}^2
\end{smallmatrix}\right)  \equiv  -I \mod F_{2n-1} \,.
$$
By Theorem \ref{t:staci}, the number $q=F_{2n-1}$ satisfies the Midy property. 

\item Let $q =F_{2n}$.   Assume for contradiction that $F_{2n}$ has the Midy property, i.e.\ by Theorem~\ref{t:Nutne} there exists $N \in \N, N>0$  such that 
\begin{equation}\label{eq:cislo}
C^N = \left(\begin{smallmatrix}
    F_{N-1}& F_{N}\\
    F_{N}& F_{N+1}
\end{smallmatrix}\right) \equiv  -I \mod F_{2n}.    
\end{equation}
In particular, $F_N\equiv 0\mod F_{2n}$, i.e.\ $F_{2n}$ is a divisor of $F_N$. From the well known fact that $F_m$ divides $F_r$ if and only if $m$ divides $r$ we derive that  $N =2n\ell$ for some $\ell \in \N$.  
Using \eqref{eq:determinantModulo}  and \eqref{eq:MaticeModulo} we obtain that   
$$
C^{2n} \equiv \left(\begin{smallmatrix}
    F_{2n-1}& 0\\
    0& F_{2n-1}
\end{smallmatrix}\right) \mod q \quad \text{and} \quad  C^{4n} \equiv \left(\begin{smallmatrix}
    F^2_{2n-1}& 0\\
    0& F^2_{2n-1}
\end{smallmatrix}\right) \equiv I \mod q 
$$
Thus $C^N = C^{2n\ell}\equiv I$ for $\ell$ even and $C^N \equiv  C^{2n}$ for $\ell$ odd. Therefore~\eqref{eq:cislo} implies that 
$$
C^N=C^{2n}\equiv
\left(\begin{smallmatrix}
    F_{2n-1}& 0\\
    0& F_{2n-1}
\end{smallmatrix}\right) \equiv -I \mod q,
$$ 
in particular, $F_{2n-1}\equiv -1 \mod F_{2n}$.

From the assumption, we have that $n \geq 3$, therefore $2\leq F_{2n-1} \leq  F_{2n}-2$, and hence $ F_{2n-1} \not\equiv \pm 1 \mod F_{2n}$. Consequently, $C^{2n}\not\equiv - I \mod q$. In summary, for every exponent $N$ we have  $C^N \not\equiv - I \mod q$.  Thus  $F_{2n}$ does not satisfy the Midy property.     
\end{enumerate}
\vspace{-1cm}
\end{proof}

In view of the sufficient condition given in Theorem~\ref{t:staci}, the knowledge of divisibility of Fibonacci numbers will be crucial. 
For a positive given $m\in\N$ denote $a(m)$ the smallest positive integer such that $m$ divides $F_{a(m)}$. The sequence $(a(m))_{m\geq 0}$ is registered in Sloane's On-Line Encyclopedia of Integer Sequences under the code A001177. 

It can be shown (e.g.~\cite{Halton1966}) that $m$ divides $F_n$ if and only if $a(m)$ divides $n$. By Item (1) of Corollary~\ref{cor:FibonacciLicheSude}, this directly implies the following.

\begin{corollary}\label{cor:a(q)odd}
Let $q>2$ be an integer  such that $a(q)$ is odd. Then $q$ satisfies the Midy property in base $\tau$.     
\end{corollary}

According to our knowledge, the description of odd values in the sequence $(a(m))_{m\geq 0}$ is not known.
Among the first 70 members of the sequence $(a(m))_{m\geq 0}$
the following are odd,
\begin{gather*}
a(5)=5, \ a(10)=15, \ a(13)=7, \ a(17)=9, \ a(25)=25, \ a(26)=21,\\ a(34)=9, \ a(37)=19, \ a(50)=75, \ a(53)=27, \ a(61)=15, \ a(65)=35.       
\end{gather*}

In the following section, we inspect the  Midy property in 
base $\tau$ for all prime denominators $q$. We will see that the necessary condition given in Corollary~\ref{cor:a(q)odd} is not sufficient, since for example $a(7)=8$ and still $7$ has the Midy property (cf.\ Theorem~\ref{thm: 2Mod5}).

On the other hand,  Corollary~\ref{cor:a(q)odd} decides about the Midy property of some non-primes, such as $10,25,26,34,50,65$.


\section{The Midy property of prime numbers}\label{sec:primes}

In order to give characterisation of primes satisfying the Midy property for the golden ratio base, we will need to work in the finite field $\mathbb{Z}_q$. Let us recall some facts from finite fields. We say that $a$ is a quadratic residue $\!\!\!\mod q$, if there exist $b\in\mathbb{Z}_q$ such that $a\equiv b^2 \mod q$. The Legendre symbol is a multiplicative function defined as
$$
\Big(\frac{a}{q}\Big) = \left\{\begin{array}{rl}
1 & \text{ if  $a$ is a quadratic residue}\!\!\mod q \text{ and } a\not\equiv 0 \!\!\mod q\\
-1 & \text{ if  $a$ is a quadratic  non-residue}\!\!\mod q,\\
0 & \text{ if }a\equiv 0 \!\!\mod q.
\end{array}\right.
$$
We also use the quadratic reciprocity law. For distinct odd primes $q_1,q_2$, we have
$$
\Big(\frac{q_1}{q_2}\Big)\cdot \Big(\frac{q_2}{q_1}\Big) = (-1)^{\frac{q_1-1}{2}\frac{q_2-1}{2}}.
$$
In particular, we derive that $5$ is a quadratic residue$\mod q$ if and only if $q=5$ or $q$ is a quadratic residue$\mod 5$, which happens exactly for $q\equiv \pm1 \mod 5$.

Halton~\cite{Halton1966} derived the following result about the value $a(q)$ for prime $q$:
If $q$ is an odd prime, then $a(q)$ divides $q-\Big(\frac{5}{q}\Big)$. By the above knowledge of quadratic residues, this amounts to saying that
\begin{equation} \label{eq:halton}
\begin{aligned}
a(q) \text{ divides } q-1 \quad \text{ if }\ q\equiv \pm1 \mod 5, \\
a(q) \text{ divides } q+1 \quad \text{ if }\ q\equiv \pm2 \mod 5.     
\end{aligned}
\end{equation}

\begin{lemma}\label{lem:odmocninaMatice} 
Let $\mathbb{F}$ be a field and let $A \in \mathbb{F}^{2\times 2}$ be a matrix such that $A^2 = I$. Then 
\begin{itemize}
    \item[(i)] either  $\det A =1$ and $A=\pm I$ 
\item[(ii)] or $\det A =-1$ and the matrix $A$
is similar to 
$D=\left(\begin{smallmatrix}1 & 0\\ 0& -1\end{smallmatrix} \right)$, 
    i.e. $A = R^{-1}DR$ for some non-singular matrix $R \in \mathbb{F}^{2\times 2}$.
\end{itemize}
\end{lemma}

\begin{proof}
    Since $A^2-I= (A-I)(A+I) = \Theta$,  
    \begin{itemize}
   \item[(i)]  either one of the matrices $A-I$
 and $A+I$ is the zero matrix $\Theta$, i.e.\ $A = \pm I$, and in this case $\det A =1$, 

 \item[(ii)] or none of the matrices is the zero matrix. In that case both $A-I$ and $A+I$ are singular, equivalently, have $0$ as an eigenvalue. This implies that the matrix  $A$ has two different eigenvalues, $1$ and $-1$, and thus is  diagonalisable. As $\det A$ is the product of eigenvalues, we have $\det A = -1$.  
 \end{itemize} 
 \vspace{-0.85cm}
 \end{proof}

Below, we will work both in the real numbers and in the finite field $\mathbb{Z}_q$.  Equality of elements $a$ and $b$   in $\R$ will be denoted  $a=b$, whereas equality in  $\mathbb{Z}_q$  we write  $a\equiv b \mod q$. Obviously  $a=b$ implies $a\equiv b \mod q$ and not vice versa.

\begin{lemma}\label{lem:1Mod5}
Let $q \in \N $ be a prime, $q\neq 5$, and let 
$C = \left(\begin{smallmatrix} 0&1\\
1&1\end{smallmatrix}\right)$  be the companion matrix of the polynomial $X^2-X-1$. If $C^{2\ell}\equiv I \mod q$ for some odd $\ell \in \N$,  then  $5$ is a quadratic residue$\mod q$ and consequently $q \equiv \pm 1 \mod 5$.
\end{lemma}

\begin{proof} 
We use the fact that for an odd integer $\ell$, we have $\det C^\ell \equiv -1$ and we verify easily that  $C^\ell =
\left(\begin{smallmatrix}
    F_{\ell-1}& F_\ell\\
    F_{\ell}& F_{\ell+1}
\end{smallmatrix}\right) $ has the inverse   
$C^{-\ell} =
\left(\begin{smallmatrix}
    -F_{\ell-1}& F_\ell\\
    F_{\ell}&- F_{\ell+1}
\end{smallmatrix}\right) $.    From $C^{2\ell} \equiv I \mod q$ we obtain $C^\ell \equiv  C^{-\ell} \mod q$. Denoting $a:=F_{\ell+1} \equiv  -F_{\ell-1} \mod q$, we thus derive that  $F_\ell \equiv  F_{\ell+1} -F_{\ell - 1} \equiv  2a  \mod q$. 
With this, we have
$$
C^\ell \equiv \left(\begin{smallmatrix}
    -a~& 2a\\
    2a& a
\end{smallmatrix}\right)  \equiv  a\left(\begin{smallmatrix}
    -1~& 2\\
    2& 1
\end{smallmatrix}\right) \!\!\!\mod q\qquad \text{and thus} \qquad  C^{2\ell} \equiv  a^2\left(\begin{smallmatrix}
    5& 0\\
  0& 5
\end{smallmatrix}\right)\equiv  I \!\!\!\mod q.
$$
Necessarily $5 \equiv (a^{-1})^2  \mod q$. In other words, $5$ is a quadratic residue$\mod q$. By the quadratic reciprocity law, we have for the Legendre symbol $(\frac{5}{q})=(\frac{q}{5})=1$, and thus $q$ is a quadratic residue$\mod 5$. Since $q$ is a prime, necessarily $q\equiv \pm1 \mod 5$.  
\end{proof}

\begin{theorem}\label{thm: 2Mod5} 
Let $q> 2$ be a prime, $q=5$ or $q \equiv\pm 2 \mod 5$. Then $q$ has the Midy property in the golden ratio base.  
\end{theorem}
    
\begin{proof}  
We have shown in Example~\ref{ex:q=3} that $5$ satisfies the Midy property.

Suppose that  $q \equiv\pm 2 \mod 5$. For such primes, we have from~\eqref{eq:halton} that $a(q)$ divides $q+1$, and thus $F_{q+1} \equiv 0\mod q$. Therefore $a:= F_{q+2} \equiv F_{q} \mod q$. As $q+1$ is even,  $\det C^{q+1} \equiv F_{q+2}F_{q} \equiv 1 \equiv a^2 \mod q$. Equality $a^2 \equiv 1 \mod q$ implies $a\equiv \pm  1 \mod q$. Thus $C^{q+1} \equiv \pm I \mod q$ and in both cases $\det C^{q+1} \equiv 1 \mod q$.

If $C^{q+1} \equiv -I \mod q$, the proof is finished. Let us discuss the case $C^{q+1} \equiv I \mod q$. 
Denote by $r$ the minimal positive integer such that $C^r=I$.
Since $C^{2r}=I$ and $q\equiv \pm2\mod 5$, Lemma~\ref{lem:1Mod5} implies that $r$ is even, say $r=2r'$. As $C^r=C^{2r'}=I$, the same Lemma~\ref{lem:1Mod5} forces $r'=2r''$. Consider the matrix $A=C^{r'}=C^{2r''}$ with determinant $\det A = (\det C)^{2r''}=1$. By Lemma~\ref{lem:odmocninaMatice}, we have that $A=C^{r'}=\pm I$. Since $r'<r$, necessarily $C^{r'}=-I$.
Using the sufficient condition in Theorem~\ref{t:staci}, we conclude that $q$ has the Midy property for the golden ratio.
\end{proof}

\begin{theorem}\label{thm: 1Mod5} 
Let $q\in \N$ be a prime, $q \equiv \pm 1 \mod 5$. Write $q -1 = 2^k \ell$ for $k\in\N$ and odd $\ell\in \N$. Then $q$ has the Midy property in the golden ratio base if and only if the list of matrices $C^{2\ell}, C^{4\ell}, \ldots, C^{2^{k-1}\ell}$ contains  a matrix which equals  $ -I \!\!\mod q$. 
\end{theorem}

\begin{proof} 
Since $q \equiv \pm 1 \mod 5$, by quadratic reciprocity, $5$ is a quadratic residue$\mod q$, i.e.\ there exists $b \in \mathbb{Z}_q$ such that $b^2 \equiv 5 \mod q$. Denote $\lambda_1 = 2^{-1}(1+b)$  and $\lambda_2 = 2^{-1}(1-b)$. Then $\lambda_1+\lambda_2 \equiv 1 \mod q$ and $\lambda_1 \lambda_2 = 2^{-2}(1-b^2) \equiv -1 \mod q$ and thus $\lambda_1$, $\lambda_2$ are the roots of the 
characteristic polynomial 
$$
(X - \lambda_1)(X-\lambda_2) = X^2-X-1
$$
of the matrix $C$. 
Obviously,  $b \not\equiv \pm 1, 0 \mod q$. This implies $\lambda_1\not\equiv \lambda_2$ and both $\lambda_1$ and $  \lambda_2$ belong to the multiplicative group $\mathbb{Z}_q \setminus \{0\} $.  The order of the elements of this group divides $q-1$, and therefore  $\lambda_1^{q-1} \equiv \lambda_2^{q-1} \equiv 1 \mod q$. Hence there exists a non-singular matrix $R \in \mathbb{Z}_q^{2\times 2}$ such that  
$$
C \equiv  R^{-1}\left(\begin{smallmatrix}
    \lambda_1& 0\\
    0& \lambda_2
\end{smallmatrix}\right)R \mod q \qquad \text{ and }  \qquad C^{q-1} \equiv  I \mod q.
$$
In order to prove the theorem, we show two implications.

The implication $(\Leftarrow)$ follows by Theorem \ref{t:staci}. For the opposite direction
$(\Rightarrow)$ we proceed by contradiction. Assume that 
$q$ satisfies the Midy property, i.e.\  by Theorem \ref{t:Nutne},  there exists  $N \in \N$ such that $C^N \equiv  -I \mod q$. In particular, for both eigenvalues $\lambda_1,\lambda_2$ of $C$ it holds that $\lambda_1^N\equiv \lambda_2^N\equiv  -1 \mod q$.

In the same time, suppose that 
no matrix in the list $C^{2\ell}, \dots, C^{2^k\ell}$ 
is equal to $-I \mod q$. Since all matrices in the list are powers of $C^2$, their determinant is equal to $1 \mod q$. Moreover, $C^{2^k\ell}=C^{q-1}\equiv I\mod q$. 
Lemma~\ref{lem:odmocninaMatice} implies that all matrices in the list are equal to $I \mod q$. As
$C^{2\ell}\equiv I \mod q$ with $\ell$ odd and $\det C^{\ell}\equiv -1$, Lemma~\ref{lem:odmocninaMatice} forces that  
$C^\ell$ is similar to the matrix 
$\left( \begin{smallmatrix}
1& 0\\
0& -1
\end{smallmatrix}\right)$. In particular, for one of the eigenvalues of the matrix $C$, say $\lambda_1$, it holds that   
$$
\lambda_1^\ell \equiv  1 \!\!\!\mod q \quad \text{and} \quad \lambda_1^N \equiv - 1 \!\!\!\mod q.
$$  
Let $r$ be the order of $\lambda_1$ in the multiplicative group of the field $\mathbb{Z}_q$. Necessarily $r$ divides $\ell$, in particular, $r$ is odd. On the other hand $\lambda_1^{2N} \equiv  1 \mod q$, and thus $r$ divides $2N$, too. The fact  $\det C = -1$  and $\det (-I) = 1$ forces $N$ to be even.  Thus the odd number $r$ divides $N/2$. We derive that  $\lambda_1^{N/2} \equiv  1 \mod q$ and  $\lambda_1^N \equiv  1 \mod q$, as well. This is a contradiction.      
\end{proof}

If $\frac{q-1}{2}$ is odd, then the list of matrices in the previous theorem is empty, and it cannot contain $-I$. By Theorem~\ref{thm: 2Mod5} we obtain the following Corollary.

\begin{corollary}\label{cor:Mod20}  
Let $q \in \N$ be a  prime such that  $q \equiv -1 \mod 20$ or $q \equiv  11 \mod 20$. Then $q$ does not have the Midy property. Consequently, no Fibonacci number $F_{2n-1}$ with $n \geq 3$ is divisible  by such a prime $q$.  
\end{corollary}

\begin{remark}
  If  $q$ is prime,  $q \equiv \pm 1 \mod 5$,   and $\frac{q-1}{2}$ is even, then necessarily  $q \equiv 1 \mod 20$ or $q \equiv  9 \mod 20$. Among such primes some have the  Midy property in base $\tau$ and some do not have. For example,  
\begin{itemize}
    \item $41 \equiv 101 \equiv 1 \mod 20$. The prime $41$ has the Midy property, whereas $101$ has not.
    \item $109 \equiv 29 \equiv  9\mod 20$.  While $109$ has the Midy property in base $\tau$, the prime  $29$ has not. 
\end{itemize}
\end{remark}

\begin{remark} Mersenne primes are defined as primes of the form $q =2^s -1$, which forces the exponent $s$ to be also prime. Of course, not all numbers of the form $2^s-1$ are prime even if $s$ is prime. Suppose that $q=2^s-1$ is a Mersenne prime. We can derive the following conclusions about the Midy property of $q$ in base $\tau$:
\begin{itemize}
    \item  if  $s \equiv  3 \mod 4$, then $q = 2^{4k+3}-1 =8 (2^4)^k - 1 \equiv 2 \mod 5$, and by Theorem~\ref{thm: 2Mod5}, the Mersenne prime $q = 2^s-1$ has the Midy property in base $\tau$. 

    \item if  $s \equiv  1 \mod 4$, then $q = 2^{4k+1}-1 = 2 (2^4)^k - 1 \equiv 1 \mod 5$, and $q = -1 \mod 4$. This forces $q = -1 \mod 20$, and by Corollary~\ref{cor:Mod20}  the Mersenne prime $q=2^s-1$ does not have the Midy property in base $\tau$. 
\end{itemize}
Today (April 2024), 51 Mersenne primes are known. Precisely 19 of them satisfy the Midy property in base $\tau$.

\bigskip

\noindent Fermat primes are primes of the form $f_n = 2^{2^n}+1$,  where $n \in \N$. 
\begin{itemize}
\item $f_0 =3$ and $f_1 = 5 =F_5$ have the Midy property by Example~\ref{ex:q=3}.

\item If $n \geq 2$, then $f_n = (2^4)^{2^{n-2}}+1 \equiv 2 \mod 5$, and by Theorem~\ref{thm: 2Mod5}, the Fermat prime $f_n$ has the Midy property.  
\end{itemize}
Thus every Fermat primes satisfies the Midy property in base $\tau$. Unfortunately, as of today, only 5 Fermat primes are known, namely $f_n$, for $n =0,1,2,3,4$.
\end{remark}

\section{Comments}

Let us discuss  possible generalizations of our result to other bases. 

\begin{itemize}
    \item  Theorem \ref{t:staci} is stated for the golden ratio base.  
    It is likely that one can generalize it for all bases $\beta$ which are quadratic Pisot units with norm equal to $-1$, i.e.\ roots of polynomials $f(X)=X^2-mX-1$, $m\geq 1$. However, much less is known about divisibility properties of sequences defined by the linear recurrence with characteristic polynomial $f$. Some results of this kind can be found in~\cite{Renault2013}.

    On the other hand, no $q$ can satisfy the  Midy property in base $\beta$ which is a quadratic Pisot unit with norm equal to $+1$. For, it is known~\cite{HI97} that no rational number in the interval $(0,1)$ has purely periodic $\beta$-expansion.
    
    \item  The original Midy's theorem considers integer bases, which are naturally not algebraic units.  
    One can observe  $q\in \N $ with the Midy property even in non-integer 
    bases $\beta >1$ that are non-units. According to Akiyama~\cite{A98}, if the base $\beta$ is chosen to be the quadratic Pisot number with minimal polynomial $X^2-mX-n$, $m\geq n$, then every reduced fraction whose denominator is coprime to $N(\beta)=n$ has a purely periodic $\beta$-expansion. For example, let $\beta=1+\sqrt3$, i.e.\ $\beta$ is a root of
    $X^2-2X-2$. Then  
    $$
    (\tfrac45)_\beta = 0\dotp(201100100121011021112000)^\omega.
    $$
 For the two halves of the period, we have
    $$
    201100100121 + 011021112000 = 212121212121
    $$
    which is a prefix of $d^*_\beta(1) = (21)^\omega$.
    
\item  As it was already mentioned in Remark~\ref{rem:Tribonacci}, in the Tribonacci base, i.e.,  when $\beta>1$ is a root of $X^3 -X^2-X-1$, no denominator  $q$ has the Midy property.  
On the other hand, we have tested fractions in base $\beta>1$ which is a root of $X^4 - X^3 - X^2 -X -1$ and we have found fractions  $\tfrac{p}{q}$ such that 
\medskip

\centerline{$T^N(\tfrac{p}{q}) = \tfrac{q-p}{q} \ \ \ \text{and} \ \ \  T^N(\tfrac{q-p}{q}) = \tfrac{p}{q}\,,$}
\medskip
 \noindent which by Lemma~\ref{lem:IFF} implies the Midy property. 
For instance, 
 for $\tfrac{p}{q} = \tfrac{1}{5}$  the previous equalities are satisfied with $N =156$; 
for $\tfrac{p}{q}=\tfrac{1}{10}$   and $\tfrac{p}{q}=\tfrac{1}{25}$ with  $N =780$, for
$\tfrac{p}{q}=\tfrac{1}{17}$ with $N =2456$. 




\end{itemize}


\EditInfo{January 9, 2024}{April 25, 2024}{Emilie Charlier, Julien Leroy and Michel Rigo}


{\small
    \begin{thebibliography}{99}

\bibitem{AFSS10} B. Adamczewski, C. Frougny, A. Siegel, and W. Steiner. Rational numbers with purely periodic $\beta$-expansion. {\it Bull. Lond. Math. Soc.}, 42:538-552, 2010.

\bibitem{A98} S. Akiyama. Pisot numbers and greedy algorithm. In {\it Number theory (Eger, 1996)}, pages 9-21. de Gruyter, Berlin, 1998.

\bibitem{Dickson66} L. E. Dickson. {\it History of the theory of numbers. Vol. I: Divisibility and primality.} Chelsea Publishing Co., New York, 1966.

\bibitem{Ginsburg2004} B. D. Ginsburg. Midy's (nearly) secret theorem - an extension after 165 years. {\it College Math. J.}, 35(1):26-30, 2004.

\bibitem{Goodwyn1802} H. Goodwyn. Curious properties of prime numbers taken as the divisors of unity. {\it J. Natur. Philos. Chem. Arts}, 1:314-316, 1802.

\bibitem{Gupta2005} A. Gupta and B. Sury. Decimal expansion of $1/p$ and subgroup sums. {\it Integers}, 5(1):A19, 5, 2005.

\bibitem{Halton1966} J. H. Halton. On the divisibility properties of Fibonacci numbers. {\it Fibonacci Quart.}, 4:217-240, 1966.

\bibitem{HI97} M. Hama and T. Imahashi. Periodic $\beta$-expansions for certain classes of Pisot numbers. {\it Comment. Math. Univ. St. Pauli}, 46:103-116, 1997.

\bibitem{Leavitt67} W. G. Leavitt. A theorem on repeating decimals. {\it Amer. Math. Monthly}, 74:669-673, 1967.

\bibitem{Lewittes2007} J. Lewittes. Midy’s theorem for periodic decimals. {\it Integers}, 7:A2, 11, 2007.

\bibitem{Martin2007} H. W. Martin. Generalizations of Midy’s theorem on repeating decimals. {\it Integers}, 7:A3, 7, 2007.

\bibitem{Midy1836} E. Midy. {\it De Quelques Propriétés des Nombres et des Fractions Décimales Périodiques}. College of Nantes, France, 1836.

\bibitem{Parry60} W. Parry. On the $\beta$-expansions of real numbers. {\it Acta Math. Hung.}, 11(3-4):401-416, 1960.

\bibitem{Renault2013} M. Renault. The period, rank, and order of the $(a,b)$-Fibonacci sequence mod $m$. {\it Math. Mag.}, 86(5):372-380, 2013.

\bibitem{Ross2010} K. A. Ross. Repeating decimals: a period piece. {\it Math. Mag.}, 83(1):33-45, 2010.

\bibitem{Renyi57} A. Rényi. Representations for real numbers and their ergodic properties. {\it Acta Math. Hung.}, 8:477-493, 1957.

\bibitem{Schmidt80} K. Schmidt. On periodic expansions of Pisot numbers and Salem numbers. {\it Bull. Lond.
Math. Soc.}, 12(4):269-278, 1980.

\bibitem{Shrader-Frechette78} M. Shrader-Frechette. Complementary rational numbers. {\it Math. Mag.}, 51(2):90-98, 1978.
\end{thebibliography}
}
\end{document}